\documentclass{amsart}

\usepackage{amsthm,lineno}
\usepackage{mathtools}
\usepackage[markup=underlined]{changes}
\usepackage{amssymb}
\usepackage{verbatim}
\usepackage{amsmath}
\usepackage[colorlinks=true, pdfstartview=FitV, linkcolor=blue,citecolor=red, urlcolor=blue]{hyperref}
\usepackage{mathtools}
\usepackage{theoremref}

\newtheorem{theorem}{Theorem}[section]

\newtheorem{corollary}[theorem]{Corollary}
\newtheorem{lemma}[theorem]{Lemma}
\theoremstyle{definition}

\theoremstyle{remark}

\newtheorem{conjecture}[theorem]{Conjecture}

\newcommand\be{\begin{equation}}
\newcommand\ee{\end{equation}}

\numberwithin{equation}{section}
\definechangesauthor[name=CK, color=orange]{CK}
\definechangesauthor[name=TA, color=magenta]{TA}
\definechangesauthor[name=CK, color=green]{JF}
\usepackage{todonotes}
\setremarkmarkup{\todo[color=Changes@Color#1!20,size=\scriptsize]{#1: #2}}

\newcommand{\QQ}{\mathbb {Q}}

\newcommand{\OO}{{\mathcal O}}

\newcommand{\edv}{\mathrel\Vert} 

\def\D{\Delta}

\def\sl2{\textbf{SL}(2,\mathbb{\OO})}
\def\psl2{\textbf{PSL}(2,\mathbb{\OO})}
\def\Psl2{\textbf{PSL}(2,\mathbb{\OO})}
\def\Pgl2{\textbf{PGL}(2,\mathbb{\OO})}

\title[Eisenstein Series and zeta functions of binary Hermitian forms]{Eisenstein Series whose Fourier coefficients are zeta functions of binary Hermitian forms}

\author{Jorge Fl\'orez}
\address{Department of Mathematics, Borough of Manhattan Community College, City University of New York, 199 Chambers Street, New York, NY 10007, USA}
\email{jflorez@bmcc.cuny.edu}

\author{Cihan Karabulut}
\address{Department of Mathematics, William Paterson University, New Jersey 07470, USA} \email{karabulutc@wpunj.edu}
\thanks{The second author was partly supported by Assigned Release Time (ART) program for research from William Paterson University}

\author{An Hoa Vu}
\address{Department of Mathematics, The Graduate Center, City University of New York, New York, NY 10016, USA}
\email{avu@gradcenter.cuny.edu}



\date{\today}

\begin{document}
	
\begin{abstract}
In this paper we investigate a result of Ueno on the modularity of generating series associated to the zeta functions of binary Hermitian forms previously studied by Elstrodt et al. We improve his result by showing that the generating series are Eisenstein series. As a consequence we obtain an explicit formula for the special values of zeta functions associated with binary Hermitian forms.
\end{abstract}
	
\maketitle

\tableofcontents

\section{Introduction}

In \cite{Cohen}, Cohen constructed modular forms whose Fourier coefficients are given by finite sums of Dirichlet $L$-series evaluated at integral arguments. As an application, he derived a number of formulas analogous to various classical class number relations discovered by Kronecker, Hurwitz, Selberg and Eichler. Zagier (\cite{Za1}) later  generalized the results of Cohen by considering infinite sums of zeta functions attached to binary quadratic forms evaluated at arbitrary complex arguments as Fourier coefficients. Similar to Zagier's construction, Ueno (\cite{U}) considered the generating series of zeta functions attached to binary Hermitian forms and showed that these infinite series are also modular forms on a congruence subgroup of the full modular group. He obtained his result using the theory of prehomogeneous vectors spaces and Weil's Converse Theorem. 

In this paper, we show that the modular forms constructed in \cite{U} are Eisenstein series and as a consequence we get a simple arithmetic expression for zeta functions attached to binary Hermitian forms. 

Let $K$ be an imaginary quadratic field with discriminant $D<0$ and let $\OO$ be its ring of integers. We let $\OO^*=\frac{i}{\sqrt{D}}\OO$ be the inverse different of $K$. For $\D, n \in \mathbb{Z} $, set
\begin{align}
r(\D,n)&:=\#\{\beta \in \OO/n\OO \ |\ \beta\bar{\beta}\equiv \D\ (\text{mod}\ n)\}\ \text{and}\\
r^{*}(\D,n)&:=\#\{\beta \in \OO^*/n\OO \ |\ |D|\beta\bar{\beta}\equiv \D\ (\text{mod}\ n|D|)\},
\end{align}
where $\#S$ denotes the cardinality of a set $S$. Then we define the following two zeta functions
\begin{align}
Z(\D,s)&:=\sum_{n=1}^{\infty}\frac{r(\D,n)}{n^{s+1}},\\
Z^*(\D,s)&:=\sum_{n=1}^{\infty}\frac{r^*(\D,n)}{n^{s+1}}.
\end{align}

The zeta function $Z(\D,s)$ is studied by Elstrodt, Grunewald and Mennicke (\cite{EGM2}) in connection with representation numbers of binary Hermitian forms with coefficients in $\OO$. They showed that 
\begin{equation}\label{EGM-zeta}
	Z(\D,s)= \begin{cases} 
		\zeta_{K}(s)L(\chi_{D},s+1)^{-1}&\text{if }  \D=0,\\
		\theta(-\D,s)\zeta_{\mathbb{Q}}(s)L(\chi_{D},s+1)^{-1} &\text{if }  \D\neq 0, \\
	\end{cases}
\end{equation}
where $\zeta_{K}(s)$ denotes the usual zeta function of $K$ and $\theta(\D,s)$ is a finite Euler product given by 
\begin{equation}\label{eulerpro}
	\theta(\D,s)=\prod_{p|D\D}R_p(\D,p^{-1-s})
\end{equation}
with
\begin{equation}
	R_p(\D,X)=\begin{cases}
		\frac{1-\left(\left(\frac{D}{p}\right)(pX)\right)^{t+1}}{1-\left(\frac{D}{p}\right)pX}&\text{for }  p\nmid D,\ p^t\edv \D,\\
		1+\left(\frac{-|D_0|^t\D_0}{p}\right)(pX)^{t+1}&\text{for }  p\mid D,\ p\neq 2,\ p^t\edv \D,\\
		1+\left(\frac{8}{\D_0D^t_2}\right)(2X)^{t+3}&\text{for } p=2,\  4\mid D,\ D_{1}\equiv 2(8),\ 2^t\edv \D,\\
		1-\left(\frac{-8}{\D_0D^t_2}\right)(2X)^{t+3}&\text{for } p=2,\  4\mid D,\ D_{1}\equiv 6(8),\ 2^t\edv \D,\\
		1-\left(\frac{-4}{\D_0D^t_2}\right)(2X)^{t+2}&\text{for } p=2,\  4\mid D,\ D_{1}\equiv 3\ \text{or}\ 7(8),\ 2^t\edv \D,\\
	\end{cases}
	\label{eq:defn_Rp}
\end{equation}
where $D_{0}:=D/p$, $\D_0:=p^{-t}\D$, and  where for $D\equiv0\pmod4$ 

$$D_{1}:=\frac{D}{4}, \quad \quad D_{2}:=\begin{cases}
-\frac{D_{1}}{2}&\text{if } D_{1}\equiv2\pmod4,\\
\frac{1-D_{1}}{2}&\text{if } D_{1}\equiv3\pmod4.
\end{cases}$$

We remark here that there is a minor mistake in Ueno's paper \cite{U} where his $Z(\Delta, s)$ is actually $Z(-\Delta, s)$ in \cite{EGM2}, Definition 2.1 and Equation (2.12).

Let $k$ be a positive integer. For $j=1,2$ Ueno defined the generating function of $Z(\D, s)$ and $Z^*(\D, s)$, respectively, as follows
\begin{align}
f_{k,D}^{j}(\tau)&:= \frac{(-1)^{k + 1}|D|^{\frac{1}{2}}\zeta(2k)\Gamma(2k+1)}{(2\pi)^{2k+1}}+\sum_{\D\geq1}(-1)^j\D^{2k}Z((-1)^{j-1}\D,2k)q^{\D},\\
g_{k,D}^{j}(\tau)&:= \frac{(-1)^{j + k+ 1}\,|D|^{\frac{1}{2}+2k}\zeta(2k)\Gamma(2k+1)}{(2\pi)^{2k+1}}+\sum_{\D\geq1} \D^{2k} Z^*((-1)^{j-1}\D,2k)q^{\D},
\end{align}
where $q := e^{2\pi i \tau}$ and $\tau \in \mathbb{H}:=\{z\in \mathbb{ C}\,|\,\Im(z)>0\}$ (for simplicity, we have normalized Ueno's $g_{k,D}^j$ by $i |D|^{-k}$). Ueno \cite[Theorem 4.4]{U} shows that $f_{k,D}^{j}$ and $g_{k,D}^{j}$ belong to $M_{2k+1}(\Gamma_0(|D|),\chi_{D})$ and satisfy 

\begin{equation}\label{frickef&g}
g_{k,D}^{j}(\tau)=(|D|\tau)^{-(2k+1)}f_{k,D}^{j}\left(\frac{-1}{|D|\tau}\right),
\end{equation}
 where $M_{2k+1}(\Gamma_0(|D|),\chi_{D})$ is the space of modular forms of weight $2k+1$ and character $\chi_{D}:=\left(\frac{D}{\cdot}\right)$ on  $\Gamma_0(|D|)$.

In this paper, we refine Ueno's result, which was obtained using Weil's Converse Theorem, by showing that the modular form $f_{k,D}^{j}$ and $g_{k,D}^{j}$ are in fact Eisenstein series (see \thref{thm:UenoFormIsEisensteinSeries}).

This paper is organized as follows. In section 2, we state the main results precisely. In section 3, we introduce the normalized versions of $f_{k,D}^{j}$ and $g_{k,D}^{j}$. In section 4, we give the proof of the main results of this paper.

\section{The Main Results}
We briefly recall some definitions and key facts about modular forms on the congruence subgroup $\Gamma_0(N)$ as discussed in \cite{CohSt}.
Let $M_{k}(\Gamma_0(N),\chi)$ be the space holomorphic modular forms on $\Gamma_0(N)$ of weight $k$ and character $\chi$.
We have a direct sum decomposition
\begin{equation}
M_{k}(\Gamma_0(N),\chi) = S_k(\Gamma_0(N),\chi) \oplus \mathcal{E}_k(\Gamma_0(N),\chi)
\end{equation}
of $\mathbb{ C}$-vector spaces, where $S_k(\Gamma_0(N),\chi)$ denotes the subspace of cusp forms and $\mathcal{E}_k(\Gamma_0(N),\chi)$ denotes its orthogonal complement with respect to the Petersson inner product. 

Let us now define the Eisenstein series that we work with in this paper. Let $\chi_1$ and $\chi_2$ be two Dirichlet characters modulo $N_1$ and $N_2$ respectively, set $N = N_1 N_2$ and view $\chi = \chi_1\chi_2$ as a character modulo $N$, and let $k\geq3$ be an integer. Then define 
\begin{equation}\label{eis}
G_k(\chi_{1},\chi_{2})(\tau):=\frac{1}{2}\sideset{}{'}\sum_{N_1|c}\frac{\overline{\chi_{1}(d)}\chi_{2}(c/N_1)}{(c\tau+d)^k},
\end{equation} 
where the sum is over all pairs $(c,d) \in \mathbb{Z}\times\mathbb{Z}$ with $(c,d)\neq 0,$ with the additional condition $N_1|c$. If $\chi_1$ is primitive then the series $G_k(\chi_{1},\chi_{2})$ belongs to the space $M_{k}(\Gamma_0(N),\chi)$ and has the Fourier expansion (see \cite[ Corollary 8.5.5]{CohSt})
\begin{equation}
G_k(\chi_{1},\chi_{2})(\tau)=\delta_{N_2,1}L(\overline{\chi_{1}},k)+\left(\frac{-2\pi i}{N_1}\right)^k\frac{\mathfrak{g}(\overline{\chi_{1}})}{(k-1)!}\sum_{n\geq 1}\sigma_{k-1}(\chi_{1},\chi_2;n)q^{n},
\end{equation}
where $\delta_{i,j}$ is the Kronecker delta function, $\mathfrak{g}(\chi_{1})$ is the Gauss sum corresponding to the character $\chi_1$ and $\sigma_t(\chi_1,\chi_{2};n)$ is the twisted divisor sum defined as  
\begin{equation}
\sigma_t(\chi_1,\chi_{2};n):=\sum_{d|n,\,d>0}\chi_1(d)\chi_2(n/d)d^t.
\end{equation}
If $\chi_2$ is trivial, we sometimes write $\sigma_t(\chi_1; n)$.

Similar to the case of the full modular group the series, $G_k(\chi_{1},\chi_{2})$ is normalized to get 
\begin{align}
E_{k}(\chi_1,\chi_2)(\tau)&:=\left(\frac{N_1}{-2\pi i}\right)^k\frac{(k-1)!}{\mathfrak{g}(\overline{\chi_{1}})}G_k(\chi_{1},\chi_{2})(\tau) \notag\\
&=\delta_{N_2,1}\frac{L(\chi_1,1-k)}{2}+ \sum_{n\geq1} \sigma_{k-1}(\chi_{1},\chi_2;n)q^{n}\label{eis-normalized}.
\end{align}
As $\chi_1$ ranges through all the primitive characters the Eisenstein series $E_{k}(\chi_1,\chi_2)$ (or equivalently, the non-normalized series $G_{k}(\chi_1,\chi_2)$)  form a basis of the subspace $\mathcal{E}_k(\Gamma_0(N),\chi)$ (see \cite[ Theorem 8.5.17]{CohSt}). 
  
 Finally, we recall the \textit{Fricke involution}  $W_N:=\begin{psmallmatrix}
 0&-1\\
 N&1
 \end{psmallmatrix}$ on a modular form $f\in M_{k}(\Gamma_0(N),\chi)$ as 
 \begin{eqnarray}
 (f|_kW_N)(\tau):=N^{-\frac{k}{2}}\tau^{-k}f\left(\frac{-1}{N\tau}\right),
 \end{eqnarray} where $|_k$ is the usual \textit{slash operator} of weight $k$ acting on the space $M_{k}(\Gamma_0(N),\chi)$. The Fricke involution of $G_{k}(\chi_1,\chi_2)$ is given by (see \cite[Proposition 8.5.3]{CohSt})
\begin{equation}\label{frickeG}
G_{k}(\chi_1,\chi_2)|_kW_N(\tau)=\chi_{2}(-1)\left(\frac{N_2}{N_1}\right)^{\frac{k}{2}}G_{k}(\overline{\chi_2},\overline{\chi_1})(\tau),
\end{equation}
which also implies that 
\begin{equation}\label{frickeE}
E_{k}(\chi_1,\chi_2)|_kW_N(\tau)=\chi_{2}(-1)\left(\frac{N_2}{N_1}\right)^{\frac{k}{2}}E_{k}(\overline{\chi_2},\overline{\chi_1})(\tau).
\end{equation}

\newcommand{\FundDisc}{\mathfrak{D}}

We are now ready to state the main result of this paper. Let $\FundDisc \subset \mathbb{Z}$ denote the subset of all fundamental discriminants and set 
\[
F_D := \left\{(D_1, D_2) \in \FundDisc^2 \;|\; D_1 D_2=D \text{ and } \gcd(D_1, D_2) = 1 \right\}.
\]

\begin{theorem}
Let $D < 0$ be the discriminant of $K$ and let $\chi_{D}$ be the quadratic character associated to $K$ then for all $j \in \{0, 1\}$, $k \in \mathbb{Z}_{\geq 1}$ and $\Delta \geq 1$, one has
\begin{equation}\label{eq:theta_vs_rhs}
f_{k, D}^{j}(\tau) = C_{k,D} \sum_{(D_1,D_2) \in F_D} |D_2|^{2k} \chi_{D_2}((-1)^{j-1}) E_{2k+1}\left(\chi_{D_1},\chi_{D_2}\right)(\tau),
\end{equation}
\begin{equation}\label{g_kD}
g_{k, D}^{j}(\tau) = C_{k,D} \sum_{(D_1,D_2) \in F_D} \frac{ |D_2|^{2k} \chi_{D_2}((-1)^{j})}{|D_1|^{2k+1}} E_{2k+1}\left(\chi_{D_2},\chi_{D_1}\right)(\tau),
\end{equation}
where
\begin{equation}\label{C_kd}
C_{k, D}:= \frac{2\,(-1)^{k+1}|D|^{\frac{1}{2}}\zeta(2k)\Gamma(2k+1)}
{(2\pi)^{2k+1}L(\chi_D,-2k)}.
\end{equation}
In particular, $f_{k, D}^{j}(\tau)$ and  $g_{k, D}^{j}(\tau)$ belong to $\mathcal{E}_{2k+1}(\Gamma_0(|D|),\chi_D)$.
	\thlabel{thm:UenoFormIsEisensteinSeries}
\end{theorem}

As a consequence of this theorem we immediately get the following simple formulas for the special values of $Z(\D,s)$ and $Z^*(\D,s)$ at positive even integers.

\begin{corollary}
For all $j \in \{0, 1\}$, $k\in \mathbb{Z}_{\geq 1}$ and $\Delta \geq 1$, we have
\begin{equation}
    (-1)^j \D^{2k} Z((-1)^{j-1}\D,2k)= C_{k, D} \sum_{(D_1,D_2) \in F_D}  |D_2|^{2k}\chi_{D_2}((-1)^{j-1}) \sigma_{2k}\left(\chi_{D_1},\chi_{D_2};\D\right),
    \label{eq:FourierCoeffUenoAsDivisorSum}
\end{equation}
\begin{equation}
\D^{2k} Z^*((-1)^{j-1}\D,2k)= C_{k, D} \sum_{(D_1,D_2) \in F_D} \frac{ |D_2|^{2k} \chi_{D_2}((-1)^{j})}{|D_1|^{2k+1}} \sigma_{2k}\left(\chi_{D_2},\chi_{D_1};\D\right).
\label{eq:FourierCoeffUenoAsDivisorSum-star}
\end{equation}

\end{corollary}

\section{Normalized Ueno's modular forms}

Instead of $f_{k, D}^{j}$, we shall work with the normalized version
\[
F_{k, D}^{j}(\tau) := \frac{L(\chi_D,2k+1)}{\zeta(2k)} f_{k, D}^{j}(\tau)
\]
that does not include the transcendental factors (i.e. those involving values of zeta function and $L$-function).

On the one hand, one has by \eqref{EGM-zeta}
\[
F_{k, D}^{j}(z) = \frac{(-1)^{k + 1}|D|^{\frac{1}{2}}\Gamma(2k+1)L(\chi_D,2k+1)}{(2\pi)^{2k+1}}
 +\sum_{\Delta\geq 1} (-1)^j
        \Delta^{2k}\theta((-1)^j \Delta,2k) q^{\Delta}.
\]

On the other hand, the functional equation of $L(\chi_D, s)$ (see \cite[Page 30]{Washington}) applied at $s = 2k+1$ gives
\[
\Gamma(2k+1)(-1)^kL(\chi_D,2k+1) = \frac{\mathfrak{g}(\chi_D)}{2i}\left(\frac{2\pi}{\mathfrak{f}(\chi_D)}\right)^{2k+1}L(\chi_D,-2k).
\]
So multiplying both sides by $\sqrt{|D|}$, and using the fact that $\mathfrak{g}(\chi_D)=i\sqrt{|D|}$, for $D < 0$, and $\mathfrak{f}(\chi_D)=|D|$, we obtain
\[
\frac{(-1)^k \sqrt{|D|}\Gamma(2k+1)L(\chi_D,2k+1)}{(2\pi)^{2k+1}}=\frac{L(\chi_D,-2k)}{2|D|^{2k}},
\]
or equivalently, after multiplying both sides by $\zeta(2k)$,
\[
\frac{(-1)^k\sqrt{|D|}\zeta(2k)\Gamma(2k+1)}{(2\pi)^{2k+1}L(\chi_D,-2k)}=\frac{\zeta(2k)}{2|D|^{2k}L(\chi_D,2k+1)}.
\]
Therefore,  the constant $C_{k, D}$ from the previous section can be described as
\[
C_{k, D} = \frac{-\zeta(2k)}{|D|^{2k}L(\chi_D,2k+1)}.
\]

In this normalized presentation, the equation \eqref{eq:FourierCoeffUenoAsDivisorSum} that we want to prove is equivalent to
\begin{equation}
    (-1)^j\Delta^{2k}\theta((-1)^j \Delta,2k)
=  \frac{-1}{|D|^{2k}}\sum_{(D_1,D_2) \in F_D}  |D_2|^{2k}\chi_{D_2}((-1)^{j-1})\sigma_{2k}\left(\chi_{D_1},\chi_{D_2};\D\right)
    \label{eq:theta_vs_rhs}
\end{equation}
for $\Delta \geq 1$ and $k \geq 1$, and \thref{thm:UenoFormIsEisensteinSeries} is equivalent to the identity
\begin{equation}
F_{k, D}^{j}(z) = -\frac{1}{|D|^{2k}}
      \sum_{(D_1,D_2) \in F_D}
         |D_2|^{2k} \chi_{D_2}((-1)^{j-1}) E_{2k+1}(\chi_{D_1},\chi_{D_2}).
      \label{eq:NormalizedUenoAsEisenstein}
\end{equation}

\section{Proof of \thref{thm:UenoFormIsEisensteinSeries} }

According to Theorem 4.6.8 in \cite{Miyake}, two modular forms $g_1, g_2 \in M_k(|D|, \chi_D)$ equal if for some integer $L$, their $n$-th Fourier coefficients $a_n(g_1)$ and $a_n(g_2)$ match for every natural number $n$ such that $(n, L) = 1$.
Thus, to show our identity between the two modular forms \eqref{eq:NormalizedUenoAsEisenstein}, it suffices to prove \eqref{eq:theta_vs_rhs} for every $\Delta > 0$ such that $(\Delta, 2D) = 1$.
To analyze the left hand side, let us denote
\begin{align*}
\theta_0(\Delta, s) &:= \prod_{p|D} R_p(\D,p^{-1-s}),\ \text{and}\\ 
\theta_1(\Delta, s) &:= \prod_{p|\D} R_p(\D,p^{-1-s})
= \prod_{p|\D} \frac{1-\left(\left(\frac{D}{p}\right)(p^{-s})\right)^{v_p(\D)+1}}{1-\left(\frac{D}{p}\right)p^{-s}}.
\end{align*}
We have
\[
\theta(\Delta, s) = \theta_0(\Delta, s) \theta_1(\Delta, s)
\]
under the assumption $(\Delta, D) = 1$. Observe that
\begin{equation}\label{theta1}
\theta_1(\Delta, s) = \sigma_{-s}(\chi_D; \Delta),
\end{equation}
where $\sigma_{-s}(\chi_D; \Delta) = \sum_{d | \Delta} \chi_D(d) d^{-s}$ is a twisted divisor sum, which follows from 
\[
\theta_1(p^t, s) = \sum_{\ell = 0}^{t} \left(\chi_D(p) p^{-s})\right)^\ell = \sum_{\ell = 0}^{t} \chi_D(p^\ell) (p^{\ell})^{-s}
\]
and the fact that $\theta_1$ is multiplicative in $\Delta$.
Here, let us make a convention that summation over $d | \Delta$ implicitly means $d > 0$ (so that $d^s$ always makes sense) whence $\sigma_{-s}(\chi_D, \Delta) = \sigma_{-s}(\chi_D, -\Delta)$.
Also note that we have a functional equation
\begin{equation}\label{funcsigma}
|\Delta|^s \; \sigma_{-s}(\chi_D, \Delta)= \chi_D(|\Delta|) \sigma_s(\chi_D^{-1}, \Delta).
\end{equation} It is worth noting that the equation \eqref{funcsigma} holds for any $\chi$ as long as the modulus of $\chi$ and $\Delta$ are relatively prime.
On the RHS of \eqref{eq:theta_vs_rhs}, one has
\begin{align*}
& \sum_{(D_1,D_2) \in F_D}  |D_2|^{2k}\chi_{D_2}((-1)^{j-1})\sigma_{2k}\left(\chi_{D_1},\chi_{D_2};\D\right)\\
= & \sum_{(D_1, D_2) \in F_D}  |D_2|^{2k} \sigma_{2k}\left(\chi_{D_1},\chi_{D_2};(-1)^{j-1}\D\right) \\
= & \sum_{(D_1, D_2) \in F_D} |D_2|^{2k} \sum_{d | \Delta} \chi_{D_1}(d) \chi_{D_2}((-1)^{j-1}\Delta/d) d^{2k}\\
= & \sum_{(D_1, D_2) \in F_D} |D_2|^{2k} \sum_{d | \Delta} \chi_{D_1}(d) \frac{\chi_{D_2}((-1)^{j-1}\Delta)}{\chi_{D_2}(d)} d^{2k} \text{ thanks to } (\Delta, D_2) | (\Delta, D) = 1\\
= & \sum_{(D_1, D_2) \in F_D} |D_2|^{2k} \sum_{d | \Delta} \chi_{D_1}(d) \chi_{D_2}(d) \chi_{D_2}((-1)^{j-1}\Delta) d^{2k} \text{ since } \chi_{D_2}(d) \in \{\pm 1\}\\
= & \sum_{(D_1, D_2) \in F_D} \chi_{D_2}((-1)^{j-1}\Delta) |D_2|^{2k} \sum_{d | \Delta} \chi_D(d) d^{2k}\\
= & \left( \sum_{(D_1, D_2) \in F_D} \chi_{D_2}((-1)^{j-1}\Delta) |D_2|^{2k} \right) \sigma_{2k}(\chi_D; \Delta).
\end{align*}
However, recalling equation \eqref{theta1} and the functional equation \eqref{funcsigma}, we have 
\[
\Delta^{2k} \theta_1((-1)^{j-1} \Delta, 2k) = \chi_D(\Delta) \sigma_{2k}(\chi_D; \Delta),
\]
where $\Delta > 0$. Thus, it remains to show that
\begin{equation}
     \chi_D(\Delta) (-1)^{j-1}  |D|^{2k} \theta_0((-1)^j \Delta, 2k)
        = \sum_{(D_1, D_2) \in F_D} \chi_{D_2}((-1)^{j-1} \Delta) |D_2|^{2k}
\label{eq:theta0_vs_character_sum}
\end{equation}
to obtain \eqref{eq:theta_vs_rhs}.
To do that, we obtain a concrete description of the set $F_D$.

For an odd prime $p$, let
$$p^* := \left( \frac{-1}{p} \right) p = \begin{cases}
p &\text{if } p \equiv 1 \bmod 4,\\
-p &\text{if } p \equiv 3 \bmod 4,\\
\end{cases}$$
be the corresponding prime fundamental discriminant.
For any $D \in \FundDisc$, we set
$$2_D^* := \frac{D}{\prod_{\text{odd } p|D} p^*}.$$

Using basic properties of fundamental discriminants, one has the following
\begin{lemma}
\begin{enumerate}
\item Let $\tilde{D} = D/4$. One has
\begin{align*}
2_D^* &= \begin{cases}
        1 &\text{if } 2 \nmid D,\\
        -4 &\text{if } 2 | D, 8 \nmid D \text{ i.e. } \tilde{D} \equiv 3\ \text{or}\ 7(8),\\
        8 &\text{if } 8 | D, \frac{D}{8} \equiv 1 \bmod 4 \text{ i.e. } \tilde{D} \equiv 2(8),\\
        -8 &\text{if } 8 | D, \frac{D}{8} \equiv -1 \bmod 4 \text{ i.e. } \tilde{D} \equiv 6(8).\\
\end{cases}
\end{align*}

\item For any $(D_1, D_2) \in F_D$, one has
$$2^*_{D_2} = 2^*_D.$$

\item Let $\Sigma_D := \{\text{primes } p | D\}$ and $2^* := 2^*_D$. Then $F_D$ is in an one-to-one correspondence with the power set of $\Sigma_D$ where any subset $Q \subset \Sigma_D$ corresponds to the pair $((D_1(Q), D_2(Q)) \in F_D$ defined by
$$D_2(Q) := \prod_{q \in Q} q^* \qquad { and } \qquad D_1(Q) := \frac{D}{D_2(Q)} = \prod_{q \in \Sigma_D \backslash Q} q^*.$$
\end{enumerate}
\thlabel{lm:description_of_FD}
\end{lemma}

We also have the following result concerning quadratic characters
\begin{lemma}
For any $D \in \FundDisc$, one has
$$\chi_D = \prod_{\text{prime } p | D} \chi_{p^*},$$
where $2^* = 2^*_D$.
\thlabel{lm:quadratic_char_factorization}
\end{lemma}

Using \thref{lm:description_of_FD} and \thref{lm:quadratic_char_factorization}, we obtain the factorization
\begin{align*}
\sum_{(D_1, D_2) \in F_D} \chi_{D_2}((-1)^{j-1} \Delta) |D_2|^{2k} &= \prod_{p|D} (1 + \chi_{p^*}((-1)^{j-1}\Delta) |p^*|^{2k})
\end{align*}
for the right hand side of \eqref{eq:theta0_vs_character_sum}.
To match with the left hand side, we observe the following

\newcommand{\sgn}{\mathsf{sgn}}
\begin{lemma}
Suppose that $\Delta$ is relatively prime to $2D$ (not necessarily positive). Then one has
\[
   R_p(\Delta, p^{-1-s}) = 1 + \sgn(p^*) \chi_{p^*}(\Delta) |p^*|^{-s}
\]
for every prime divisor $p | D$.
Here, $\sgn$ is the sign character i.e. $\sgn(x) = 1$ if $x > 0$ and $-1$ otherwise and could also be characterized alternatively in term of Kronecker symbol $\sgn(x) = \left( \frac{x}{-1} \right)$.
\end{lemma}
\begin{proof}
Observe that for odd $\Delta$, we always have
\[
   \left(\frac{-\D}{p}\right) = \sgn(p^*) \chi_{p^*}(\Delta).
\]
Indeed, if $p \equiv 1 \bmod 4$ then
\[
   \left(\frac{-\D}{p}\right)
= \left(\frac{\D}{p}\right)
= \left(\frac{p}{\D}\right)
= \sgn(p^*) \chi_p(\D)
\]
by quadratic reciprocity, and  if $p \equiv 3 \bmod 4$ we have
\[
   \left(\frac{-\D}{p}\right)
= -\left(\frac{\D}{p}\right)
= -\left(\frac{p}{\D}\right) (-1)^{\frac{\D-1}{2}}
= -\left(\frac{-p}{\D}\right) \underbrace{\left(\frac{-1}{\D}\right) (-1)^{\frac{\D-1}{2}}}_{1 \text{ by supplementary rule}}
= -\chi_{-p}(\D).
\]
By assumption $(\Delta, D) = 1$, we have $\Delta_0 = \Delta$ and $t = v_p(\D) = 0$ for all $p | D$ in \eqref{eq:defn_Rp}.
It follows that, for odd $p$, 
\[
   R_p(\Delta, p^{-1-s})
= \left(1+\left(\frac{-\D}{p}\right)(p^{-s})\right)
= 1 + \sgn(p^*) \chi_{p^*}(\Delta) |p^*|^{-s}
\]
by definition. For $p = 2$, one likewise has
\begin{align*}
	R_2(\D, 2^{-1-s}) &= \begin{cases}
		1+\left(\frac{8}{\D}\right) (2^{-s})^{3}&\text{for } D_{1}\equiv 2(8),\\
		1-\left(\frac{-8}{\D}\right) (2^{-s})^{3}&\text{for } D_{1}\equiv 6(8),\\
		1-\left(\frac{-4}{\D}\right) (2^{-s})^{2}&\text{for } D_{1}\equiv 3\ \text{or}\ 7(8),\\
	\end{cases}\\
	&= \begin{cases}
		1+\chi_{2^*}(\D) (8^{-s}) &\text{for } D_{1}\equiv 2(8) \Rightarrow 2^* = 8,\\
		1 - \chi_{2^*}(\D) (8^{-s}) &\text{for } D_{1}\equiv 6(8)  \Rightarrow 2^* = -8,\\
		1 - \chi_{2^*}(\D) (4^{-s}) &\text{for } D_{1}\equiv 3\ \text{or}\ 7(8)  \Rightarrow 2^* = -4,\\
	\end{cases}\\
	&= 1 + \sgn(2^*) \chi_{2^*}(\D) |2^*|^{-s}.
\end{align*}
This establishes the lemma.
\end{proof}

By the lemma above and $\chi_D(-1) = -1$, we can now write
\begin{align*}
   &(-1)^{j-1} \, \chi_D(\Delta) \, |D|^{2k} \, \theta_0((-1)^j\Delta, 2k)\\
= &\chi_D((-1)^{j-1} \Delta) |D|^{2k} \prod_{p|D} (1 + \sgn(p^*) \chi_{p^*}((-1)^j \D) |p^*|^{-2k})\\
= &\prod_{p|D} \chi_{p^*}((-1)^{j-1} \Delta) |p^*|^{2k} (1 + \chi_{p^*}(-1) \chi_{p^*}((-1)^j \D) |p^*|^{-2k}) \quad \text{ for } \chi_{p^*}(-1) = \sgn(p^*)\\
= &\prod_{p|D} (\chi_{p^*}((-1)^{j-1} \Delta) |p^*|^{2k} + 1).
\end{align*}

We have thus completed the proof of \eqref{eq:theta_vs_rhs} for all $(\Delta, 2D) = 1$; hence the proof of \eqref{eq:NormalizedUenoAsEisenstein} assuming Ueno's result about $f_{k, D}^{j}$ being a modular form.

Finally, the identity (\ref{g_kD}) follows from the equation (\ref{frickef&g}) and the Fricke involution of $f_{k,D}^j$ as follows. The equation (\ref{frickef&g}) implies that 
\begin{align*}
g_{k,D}^j(\tau)
&=|D|^{-\frac{2k+1}{2}}\left(|D|^{-\frac{2k+1}{2}}\tau^{-(2k+1)}f\left(\frac{-1}{|D|\tau}\right)\right)\\
&=|D|^{-\frac{2k+1}{2}}\left(f_{k,D}^j\bigg\rvert_{2k+1}W_{|D|}\right)(\tau)\\
&=C_{k,D} \sum_{(D_1,D_2) \in F_D}  |D_2|^{2k} \chi_{D_2}((-1)^{j})|D_1|^{-(2k+1)} E_{2k+1}\left(\chi_{D_2},\chi_{D_1}\right)(\tau),
\end{align*}
which shows that $g_{k,D}^j\in \mathcal{E}_{2k+1}(\Gamma_0(|D|),\chi_D).$

\vskip 0.2in

\noindent\emph{Acknowledgments.} The authors would like the referee for a careful reading of the paper and for helpful comments. 

\bibliographystyle{alpha}
\bibliography{Zeta}

\end{document}